\documentclass[10pt]{amsart}
\usepackage{amssymb, amscd, amsmath, amsthm, epsfig, epsf, latexsym,graphicx}
\usepackage{  pb-diagram, fancyhdr, graphicx}
\newcommand{\compactlist}{\begin{list}{$\bullet$}{\setlength{\leftmargin}{1em}}}
\def\zz{{\bf Z}}

\def\qq{{\bf Q}}

\def\rr{{\bf R}}
\def\co{\colon\thinspace}

\def\calg{\mathcal{G}}

\newcommand{\fig}[2] { \includegraphics[scale=#1]{#2} }
\newcommand{\Tors}{\operatorname{Torsion}}
\newcommand{\Tor}{\operatorname{Tor}}
\newcommand{\Arf}{\operatorname{Arf}}
\newcommand{\lk}{\ell k}

\newtheorem{theorem}{Theorem}
\newtheorem{lemma}[theorem]{Lemma}
\newtheorem{corollary}[theorem]{Corollary}
\newtheorem{prop}[theorem]{Proposition}
\newtheorem{definition}[theorem]{Definition}
\newtheorem{conj}[theorem]{Conjecture}
\newtheorem{ques}[theorem]{Question}

\def\co{\colon\thinspace}

\numberwithin{equation}{section}

\begin{document}

\title{On surgery curves for genus one slice knots}
\author{Patrick M. Gilmer}
\author{Charles Livingston} 
\thanks{The first author was partially supported by    NSF-DMS-0905736}  
\thanks{The second author was partially supported by    NSF-DMS-1007196}
 \address{Patrick Gilmer: Department of Mathematics, Louisiana State University, Baton Rouge, LA 70803}
\email{gilmer@math.lsu.edu}
 \address{Charles Livingston: Department of Mathematics, Indiana University, Bloomington, IN 47405 }
\email{livingst@indiana.edu}


 \begin{abstract}  If a knot $K$ bounds a genus one Seifert surface $F \subset S^3$  and $F$ contains an essential simple closed curve $\alpha$ that has induced framing 0 and is smoothly slice, then $K$ is smoothly slice.  Conjecturally, the converse holds.  It is known that if $K$ is slice, then there are strong constraints on the algebraic concordance class of  such $\alpha$, and it was 
  thought that these constraints might imply that
 $\alpha$ is at least algebraically slice.  We present a counterexample; in the process we answer negatively a question of Cooper and relate the result to a problem of Kauffman.   Results of this paper depend on the interplay between the Casson-Gordon invariants of $K$ and algebraic invariants of $\alpha$. 
 \end{abstract}

\maketitle
 

 \section{Introduction.}
 
For $n>1$, if a smooth knotted $S^{2n-1}$ in $S^{2n+1}$ bounds an embedded disk in $B^{2n+2}$, such a smooth slicing disk can be constructed from a $2n$--manifold bounded by $K$ in $S^{2n+1}$ by  ambient surgery.  Whether the same is true for knots in $S^3$ has remained an open question for 40 years, though by the work of  Freedman~\cite{freedman-quinn},  counterexamples exist in the topological category. 
  
One well-known and simply stated conjecture~\cite[Problem 1.38]{kirby} is a special case:  the untwisted Whitehead double of a knot  $J \subset S^3$ is smoothly slice if and only if $J$ is  smoothly slice.   
More generally, if $K$ is a knot in $S^3$  that bounds a genus one Seifert surface $F$ and is algebraically slice, then up to isotopy and orientation change, there are exactly two essential simple closed curves on $F$, $J_1$ and $J_2$, with self-linking 0 with respect to the Seifert form of $F$. In this situation, we will call $J_1$ and $J_2$ surgery curves for $F$. Conjecturally, if $K$ is smoothly slice, then one of $J_1$ or $J_2$ is necessarily  smoothly slice; see  \cite[Strong Conjecture, page 226]{Ka}, for instance.  

Shortly after Casson and Gordon~\cite{CG1}~developed obstructions to  slicing algebraically slice knots, it was noticed that Casson-Gordon invariants could be expressed in terms of signature invariants  of curves on Seifert surfaces~\cite{G1,lith}. Moreover,  Casson-Gordon invariants could be interpreted in this way as  obstructions to slicing $K$ by slicing a surgery curve on a genus one Seifert surface for $K$.  Casson-Gordon invariants actually obstruct topological locally flat slice disks.

 A genus one knot $K$ is algebraically slice if   and only if it has  an Alexander polynomial of the form
 \begin{align*}\Delta_K(t)& = (m t-(m+1) )((m+1)t -m )\\ &= m(m+1)t^2 -( m^2+(m+1)^2)t +m(m+1) \end{align*}
 for some $m \ge 0$. 
Observe that if $\Delta_K$ has the form above,  then the non-negative integer $m$ is determined. For a genus one algebraically slice knot $K$, let  $m(K)$ denote this number; note that the determinant of $K$ is $(2m(K)+1)^2.$

We let $\sigma_K(t)$ denote the Levine-Tristram~\cite{levine,tristram} signature function of $K$, as defined on the unit interval $[0,1]$ and redefined to be the average of the one-sided limits at the jumps.   Casson-Gordon theory implies that if a genus-one knot $K$ is slice 
and $m(K) \ne 0$, then the signature function of one of the surgery curves satisfies strong constraints.  To state these, we make the following definition.

 \begin{definition}\label{sigconddef}  A knot $J$ satisfies the   $(m,p)$--signature conditions for integers $m>0$ and $p$ relatively prime to $m$ and $m+1$, if 
 $$\sum_{i=0}^{r-1} \sigma_{J}( ca^i/p) = 0$$
for all     
$c \in {\zz_p}^*$,
  and  $a =  \frac{m+1}{m} \mod p$,  where $r$ is the order of $a$ modulo $p$.
 \end{definition}

To get a feeling  for this summation, consider the case of $m(K)=1$ and  
$p = 73$.   In $\zz_{73}$,  the number 2 generates the multiplicative subgroup $\{1,2,4,8,16,32,64,55,37\}$.  This subgroup has 
8 cosets in the group of units $(\zz_{73})^*$.
For instance the coset containing $c = 5$ is  $\{5,7,10,14,20,28, 39,40,56  \}$.
Thus the following arises as one of the sums in the $(1,73)$--signature condition: 
$$\sigma_J(\frac{ 5}{73})  + \sigma_J(\frac{7}{73})  +\sigma_J(\frac{10}{73} )+\sigma_J(\frac{14}{73} ) +\sigma_J(\frac{20}{73})  +\sigma_J(\frac{28}{73})  +\sigma_J(\frac{39}{73}) +\sigma_J(\frac{40}{73}) + \sigma_J(\frac{56}{73}) .$$ Notice  that the cosets  appear to be fairly randomly distributed in the unit interval.  Nonetheless, as we show,  the vanishing of all such sums is not sufficient to imply the vanishing of the signature function itself.
Consider the following simple  consequence  of Theorem~\ref{thmwitt1} below.

 \begin{theorem}\label{cgcooperthm} Let $K$ be  a genus one smoothly slice knot, then  one of the surgery curves $J$ satisfies $(m(K),p)$--signature conditions for an infinite set of primes $p$.
 \end{theorem}
\noindent In his unpublished thesis~\cite{cooper}, Cooper in fact  stated a stronger result. 

 \begin{theorem}\label{cooperthm} Let $K$ be  a genus one smoothly slice knot, then  one of the surgery curves $J$ satisfies the  $(m(K),p)$--signature conditions for all $p$ relatively prime to $m$ and $m+1$.
 \end{theorem}

One quick corollary, first observed by Cooper, of either of these theorems is that for a genus one slice knot $K$ with $m(K) >0$, the integral of the signature function of one of the slice curves $J$ is 0.  This follows by summing the signature sums in the theorem over all values of $c$ to get a sum of the form $\sum_{i=1}^{p-1} \sigma_{J}( i/p) = 0$ and then noting that for large $p$, this sum approximates the integral.  (This integral condition was later seen to follow from the $L^2$-signature approach of~\cite[Theorem(1.4)]{COT}.)

Clearly, the  constraints given by these theorems are quite extensive.  One explicit question asked by Cooper is whether the fact the combined sum $\sum_{i=1}^{p-1} \sigma_{J}( i/p) = 0$ for the appropriate infinite sets of $p$ implies the vanishing of the signature function~\cite[Question (3.16)]{cooper}.   We will show that the answer is no.  In fact, the much stronger constraints given in Theorems~\ref{cgcooperthm} and~\ref{cooperthm} are not sufficient to imply the vanishing of the signature function of one of the surgery curves.  Here is the algebraic formulation of the question.

\vskip.1in

\begin{ques} Let $\sigma$ be an integer-valued step function defined on $[0,1]$ with the property that $\sigma(x) = \sigma(1-x)$ for all $x$. Assume also $\sigma(0)=\sigma(1)=0$, that  there are no jumps at points with denominator a prime power, and that $\sigma$ is equal to the average of the one-sided limits at the jumps.  Suppose that for all    $p>1$   coprime to $m$ and $m+1$, for  $G$   the multiplicative subgroup of $(\zz_p)^*$ generated by $\frac{m+1}{m}$,   and for all $n \in \zz_p$,   $$\sum_{r \in nG} \sigma (r/p) = 0.$$  Then does $\sigma(t) = 0$ for all $t$?
\end{ques}
 
For each $m>0$, the answer to the above question is emphatically no.  Let $K_{(r,s)}$ denote the $(r,s)$--cable of $K$ (that is,  $r$ longitudes, and $s$ meridians).  Let $ - K$ denote the mirror image of $K$. 
 
\begin{theorem} \label{shifted1}  Let $K$ be a knot with a non-zero signature function, and $m>0$. 
The signature function of $K_{(m,1)} \#-{K_{(m+1,1)}}$ is non-zero and satisfies the   $(m,p)$--signature conditions for all $p$ relatively prime to $m$ and $m+1$.
 \end{theorem}

We have a  perhaps nicer   family to work with in the case $m=1.$ Let $T_{r,s}$ denote the $(r,s)$--torus knot, which is the 
$(r,s)$--cable of the unknot. 

\begin{theorem} \label{shifted1'}  If $r$ is an odd number,  
and $ r \ge 3$, the signature function of $(T_{2,r})_{(2,-r)}$ is non-zero and satisfies the  $(1,p)$--signature conditions for $p$ odd.
 \end{theorem}

\begin{figure}\fig{.4}{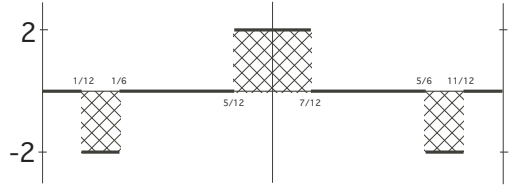}
\caption{Signature Function of $(T_{2,3})_{(2,-3)}$ which satisfies the    $(1,p)$--signature conditions for $p$ odd.} \label{sigma-function}
\end{figure}

Although  
Casson-Gordon theory provides a somewhat weaker   version of Cooper's theorem, it provides access to
the more powerful 
Witt class analogs of Theorem~\ref{cgcooperthm}, 
which carry more information than given by signatures.
Also, Casson-Gordon theory obstructs topological sliceness, whereas Cooper worked in the smooth category.
  We now describe these Witt class invariants.

 If $K$ is a knot,  let 
 $ V_t = (1-t)V+ (1-t^{-1}){V}^t$, where      $V$ is a Seifert matrix of $K$  and  $t$ is an indeterminant.   For $p$ a prime power and $j/p \in  \zz[\frac{1}{p}]/\zz $, let  $w_K(\frac{j}{p})$ denote the element represented by  $V_{e^{2 \pi i j/p}}$ in   $W( \qq(\zeta_p) )\otimes \zz_{(2)}$.  Here 
 $W( \qq(\zeta_p) )$ denotes the Witt group of hermitian forms over the field $ \qq(\zeta_p)$ and $\zz_{(2)}$ denotes $\zz$ localized at two.  An elementary proof shows that this defines a homomorphism on the concordance group.
 
\begin{definition}\label{witdef}  We say a knot $J$ satisfies the $(m,p)$--Witt conditions for integers $m>0$, $p$ relatively prime to $m$ and $m+1$, if 
 
$$\sum_{i=0}^{r-1} w_{J}[\frac {c+a i} p]=0 \in  W( \qq(\zeta_p) )\otimes \zz_{(2)},$$
for all   $c \in {\zz_p}^*$,   $a =  \frac{m+1}{m} \mod p$,  and
$r$   the order of $a$ modulo $p$.
\end{definition} 
 
If a knot $J$ satisfies the $(m,p)$--Witt conditions then it satisfies as well the $(m,p)$--signature conditions.  But the Witt conditions are stronger. For instance, one can define a discriminant invariant   on $W( \qq(\zeta_p) )\otimes \zz_{(2)}$ which is discussed in~\cite{GL2}. 

 \begin{theorem}\label{thmwitt1} Let $K$ be a genus one topologically slice knot.  There is some finite set of bad primes $P$ such  that  one of the surgery curves $J$ satisfies the $(m(K),p)$--Witt conditions  for all $p$ in the set 
$$  \{ r^n\  |\  n \in \zz_+,  r \text{ is  prime, } r \notin P, r^n \text{ divides }  (m+1)^q-(m)^q  \text{ for some prime power $q$} \}.$$
 \end{theorem} 
 
 Consider $Wh(J,n)$,  
  the $n$--twisted Whitehead double of $J.$ It is well-known that this knot is algebraically slice if and only if  $n =m(m+1)$. Moreover $m(Wh(J,m(m+1) ))= m$. It is also known that the two surgery curves for $Wh(J,m(m+1) )$  both have the isotopy type of $J \# T_{(m,m+1)}$. One can see this using the techniques discussed in~\cite[pages 214--223]{Ka}.     Using this fact, for these knots one can  sometimes remove the exceptions created by the unknown set of bad primes.
 
  \begin{theorem}\label{thmwitt2}  Let $m > 0$. If \ $Wh(J, m(m+1))$ is topologically slice,  
  then $J \#T_{(m,m+1)}$ satisfies  the $(m,p)$--Witt  conditions for all  $p$  
in the set: 
$$  \{ p\ |\  p \text{ is a prime, }  \gcd(p^2 ,  (m+1)^q-(m)^q)=p \text{ for some odd prime power $q$} \}.$$
 \end{theorem} 

Our examples of knots satisfying  $(m,p)$--signature conditions also satisfy  Witt conditions.

\begin{theorem} \label{shiftedw}  For any knot $K$ and $m >0$, 
$K_{(m,1)}\# - K_{(m+1,1)}$   
  satisfies the 
$(m,p)$--Witt conditions for all $p$ relatively prime to $m$ and $m+1$. For any odd integer $n$, $(T_{2,n})_{(2,-n)}$ satisfies the 
$(1,p)$--Witt conditions for all odd $p$.\end{theorem}

 In the next theorems, we focus on some particularly nice examples.
    
\begin{theorem}\label{mainthm}   Let $J= (T_{2,3})_{(2,-3)}$, the $(2,-3)$--cable of  trefoil knot $T_{2,3}$. Let   
$K = Wh(J ,2)$. 
\begin{enumerate}
\item   $K$ is a genus one algebraically slice knot with both surgery curves having the same knot type:  
$J$. 
\item  $J$ satisfies the  $(1,p)$--Witt conditions for all odd $p$. In particular $J$ satisfies the  $(1,p)$--signature conditions for all odd $p$.  Another consequence is that the constraints of Theorems~\ref{cooperthm},~\ref{thmwitt1}, and~\ref{thmwitt2} on $K$ are satisfied.
\item The signature function of $J$ is non-zero.
\item $\Delta_{J}(t)=(t^{-1}-1+t)(t^{-2}-1+t^2)$ does not satisfy the Fox-Milnor condition; that is, $\Delta_{J}
(t)$ cannot be written as $f(t)f(t^{-1})$ for $f(t) \in \zz[t,t^{-1}]$.
\item $\Arf{J}\ne 0$. 
\end{enumerate}
\end{theorem}
 
We do not know whether $Wh((T_{2,3})_{(2,-3)} ,2)$ is topologically locally-flat slice or smoothly slice. 
A conjecture made by  Kauffman \cite[Weak Conjecture, page 226]{Ka} ~\cite[Problem 1.52]{kirby} implies that 
 $Wh((T_{2,3})_{(2,-3)} ,2)$ is not smoothly slice since $\Arf((T_{2,3})_{(2,-3)}) \ne 0$.  Thus examples such as this one offer a route to possible counterexamples to this conjecture.

By modifying the example slightly (without changing the relevant signature function, Alexander polynomial,  Arf invariant  or even Witt class invariant), results of Hedden~\cite{he1, he2} on the Ozsv\'ath-Svab\'o invariant of cables and Whitehead doubles, obstructing sliceness becomes possible. This is described in the first part of the following theorem. We also give a second example of a knot with similar properties.

 \begin{theorem}\label{nsslice} 
 Let $J' = (T_{2,3}\#Wh(T_{2,3},0))_{(2,-3)}$.
   Then $K'= Wh(J',2)$  is not smoothly slice. Moreover the conclusions of Theorem~\ref{mainthm} hold when $K$ is replaced by $K'$ and $J$ is replaced by $J'$.
  
    Let $J''= (T_{2,3})_{(2,-3)} \# (T_{2,3})_{(2,-3)}$.  Then $K''= Wh(J'',2)$  is not smoothly slice. Moreover  conclusions  (1), (2), and (3) of Theorem~\ref{mainthm}  hold when $K$ is replaced by $K''$ and $J$ is replaced by $J''$.
 \end{theorem}

In Section~\ref{secshifted},  we outline the proofs of Theorems~\ref{shifted1},~\ref{shifted1'} and~\ref{shiftedw}.  Section~\ref{sectionnotslice} presents the proof of Theorem~\ref{nsslice} using tools from Heegaard-Floer theory.  In Section~\ref{sectioncg}  and Appendices~\ref{hlemma} and~\ref{sec2ndapproach}, we review Casson-Gordon theory and prove    Theorems~\ref{thmwitt1} and
\ref{thmwitt2}.  Similar arguments have appeared, but some depend on a theorem stated by the first author~\cite[Theorem 1]{G2}, whose proof contains a gap (shared with~\cite[Theorem (0.1)]{G1}).  We show how to modify this proof to obtain the results stated above.
In Section~\ref{restrict}, we give some restrictions on  signature functions which satisfy the $m$-signature averaging conditions.
    
 
\section{ Proofs of Theorems~\ref{shifted1},~\ref{shifted1'} and~\ref{shiftedw}}\label{secshifted}
 
Let $S$ be a finite set in $\rr/ \zz$.  
For any function $f(t)$ on $\rr/ \zz$ taking values in an abelian group, define $\mu_S(f(t)) = \sum_{s \in S}f(s)$.
We let $\phi_k\co \rr/\zz \to \rr / \zz$  denote multiplication by the integer $k$. Observe that that if $\phi_k$ is injective on $S$, then $\mu_{\phi_k(S)}(f(t)) = \mu_{S}(f(kt))$. In particular, we have the following.

\begin{lemma}\label{shift} If $S \subset \rr/\zz$ is a finite set on which $\phi_m$ and $\phi_n$ are both injective and $\phi_m(S)=\phi_n(S)$, then for all $f$, $\mu_S(f(mt) - f(nt)) = 0$.

\end{lemma}

In the current case of interest, we have an  integer  $m$, an integer $p$ relatively prime to $m(m+1)$, and an integer  $c$ representing an element in $\zz_p^*$.  We let $a =    \frac{m+1}{m} \mod p$ and   $S =  \{ \frac{ca^i}{p}\}\subset \qq/\zz$. Notice that $m a^i = (m+1) a^{i-1}$.  Thus, in this setting $\phi_m(S) = \phi_{m+1}(S)$.

\begin{corollary} With the notation of the previous paragraph, for all $f$, $$\mu_S( f( (m+1)t) - f(mt)) = 0.$$
\end{corollary}

An immediate application is the case that $f$ is the signature function of a knot $J$, in which case $f(mt)$ is the signature function of the knot $J_{m,\pm1}$.  

\vskip.1in
In the proof of Lemma \ref{shift},  it is not required that $f$ be defined on all of $\rr /\zz$, but only on the sets $S, \phi_m(S)$ and $\phi_n(S)$.  For instance, for a knot $J$ and prime power $p$, there is the function $w_J \co \{\frac{j}{p}\} \to W(\qq(\zeta_p)) \otimes \zz_{(2)}$,  
defined by  $$w_J(\frac{j}{p}) = (1-\zeta_p^j)V + (1-\zeta_p^{-j})V^t,$$ where $\zeta_p = e^{2 \pi i /p}$.

The only missing ingredient, in the proofs of Theorems~\ref{shifted1},~\ref{shifted1'} and~\ref{shiftedw}, is the following theorem.
\begin{theorem}\label{companionthm}
If $\mathbb{S}$ is a satellite of $C$ 
with orbit 
$P$ 
and winding number 
$n$, 
 then
$$ w_{\mathbb{S}}(\frac{j}{p})= w_{P}(\frac{j}{p}) + w_{C}(\frac{nj }{p}).$$
\end{theorem}

This result is very close to a result of Litherland~\cite[Theorem 1]{lith} which states that if $V_t(K) = (1-t)V +(1-t^{-1})V^t$, where $V$ is the Seifert form of $K$, then $V_t(\mathbb{S})$ is Witt equivalent to the form 
 $V_{t^n}(C) \oplus V_t(P)$ 
in the Witt group of the function  field $W(\qq(t))$.  
One would like to argue at this point that  
the  substitution of  $\zeta_p$ for $t$ defines a map $W(\qq(t)) \to
W(\qq(\zeta_p))$, and 
Theorem~\ref{companionthm} results.  Unfortunately,   
this procedure does not lead to a well-defined map $W(\qq(t)) \to W(\qq(\zeta_p))$,
as a class in $W(\qq(t))$ may be represented by a matrix whose entires 
have poles at $\zeta_p$.
  We leave it to  appendix \ref{cableappendix}  to show how this hurdle can be overcome.


\section{Smooth obstructions to slicing}\label{sectionnotslice}

In~\cite{os1} an invariant $\tau$ is defined with the property that if $K$ is smoothly slice, then $\tau(K) = 0$.  In order to apply this, we need to modify our knot  $K$ slightly.  Let $K' = Wh((T_{2,3}\#Wh(T_{2,3},0))_{(2,-3)},2)$. We show $\tau(K') =1$.

As a first step, it follows from~\cite{os1} that $\tau(T_{2,3}) = 1$.  Next, Hedden~\cite{he1} proved  that for any $J$,  $\tau(Wh(J,t)) = 1$ for all   
$t < 2 \tau(J)$.
 Thus, $\tau(Wh(T_{2,3},0))= 1$.  By additivity, $\tau(T_{2,3}\#Wh(T_{2,3},0)) = 2$. 

According to another theorem of  Hedden~\cite{he2}, if $\tau(J) = \text{genus}(J)$  then $$\tau(J_{(s,sn+1)}) = s\tau(J) +\frac{(sn)(s-1)}{2} + s-1.$$
In the case of interest to us, we have $s=2$ and $n=-2$, so $\tau(J_{(2,-3)})= 2\tau(J) -1.$  We do have  $\tau( T_{2,3}\#Wh(T_{2,3},0) ) =   \text{genus}( T_{2,3}\#Wh(T_{2,3},0) ) =2$, so, 
$$\tau( 
(  
T_{2,3}\#Wh(T_{2,3},0)
)_{(2,-3)}
)
 = 2\tau(
 T_{2,3}\#Wh(T_{2,3},0)
 ) -1 = 2(2)-1 = 3.$$

 Finally, again by Hedden's computation of $\tau$ of doubled knots, 
 $$  \tau( 
 Wh( 
 ( 
 T_{2,3} \# Wh(T_{2,3},0)
 )_{(2,-3)}
  ,t)
 ) 
 = 1$$ if  
 $t <6$.
  So in particular,  
  $  \tau( 
 Wh( 
 ( 
 T_{2,3} \# Wh(T_{2,3},0)
 )_{(2,-3)}
  ,2)
 ) 
 = 1$.

 We can also consider $K''= Wh(  (T_{2,3})_{(2,-3)} \# (T_{2,3})_{(2,-3)} , 2).$
 One has $$\tau((T_{2,3})_{(2,-3)})=1$$ using the same formula of Hedden's for cables. So 
 $$\tau( (T_{2,3})_{(2,-3)} \# (T_{2,3})_{(2,-3)} )=2.$$ Then using Hedden's formula for doubles, 
 $\tau(K'')=1.$

\section{Casson-Gordon theory}\label{sectioncg}

By a character  $\chi$ on $X$, we mean a homomorphism $\chi:H_1(X) \rightarrow  \qq/\zz.$
This is a  $d$--character if $\chi:H_1(X) \rightarrow (1/d) \zz/\zz \subset \qq/\zz.$
Given a knot $K$ and a prime power $q$,   
let $S_q$ 
denote the $q$--fold branched cyclic cover of $S^3$ along $K$. Given a $d$--character on $S_q$, Casson and Gordon~\cite{CG1} defined an invariant $\tau(K,\chi)$ taking values in $W(\qq[\zeta_d](t)) \otimes \qq.$ 
Here $W(\qq[\zeta_d](t))$ is the Witt group of Hermitian forms over $\qq[\zeta_d](t)$.  If $d$ is odd (as will be the case when $K$ is a genus one algebraically slice knot), then
$\tau(K,\chi)$ may be refined~\cite{GL1,GL2} to take values in  $W(\qq[\zeta_d](t))\otimes \zz_{(2)}.$  This refinement is useful as these Witt groups have 2--torsion.  Here is the theorem of Casson-Gordon ~\cite{CG1,CG2} which asserts
that certain $\tau(K,\chi)$ vanish when $K$ is slice. (Casson and Gordon proved this theorem for smooth slice disks, and later, based on the work  of Freedman and Quinn~\cite{freedman-quinn}, it was seen to hold in the topological locally flat category.)

 \begin{theorem}[Casson-Gordon~\cite{CG1}]\label{thmcg} Let $K$ be a slice knot  bounding a slice disk $\Delta \subset B^4$.
 Let   
 $W_q$  be the q--fold cyclic branched cover of $B^4$ over $\Delta.$  
 \begin{itemize}
 \item
 If  $\chi$ is a character  on $S_q$ of prime power order  that extends to $W_q$, then   
 $\tau(K, \chi)=0 $. 
 \item A character $\chi$ on $S_q$ extends to $W_q$ if  and only if it vanishes on $\kappa(\Delta,q)$
the kernel of $H_1(S_q) \rightarrow H_1(W_q)$.
 \item
 The kernel $\kappa(\Delta,q)$ is a metabolizer for the linking form on $H_1(S_q)$ and is invariant under the the group of covering transformations. 
  \item
The set of  characters $\chi$ on $S_q$  which extend  to  $W_q$ form a metabolizer, $\mathfrak{m}(q,\Delta)$, for the linking form on $H^1(S_q, \qq/\zz).$
 \end{itemize}
 \end{theorem}
 
   If $p$ is a prime and $G$ is an abelian group, let $G_{(p)}$ denote  the $p$--primary summand of  $G$. Note that the obstruction to sliceness given by Theorem~\ref{thmcg} can be reduced to a sequence of obstructions associated to each prime  $p$: $\tau(K,\chi)=0$ for $\chi \in \mathfrak{m}(q,\Delta)_{(p)}$.

Let $F$ be a Seifert surface   $K$.  Then  $F \cup \Delta$ bounds a 3--manifold $R \subset B^4$.   
In~\cite[Theorem 1]{G2}, the first author related $\mathfrak{m}(q,\Delta)$ to the metabolizer $H$ for Seifert form on $H_1(F)$   that arises as the kernel of  the map induced by inclusion, 
$H_1(F) \rightarrow H_1(R)/{\Tors(H_1(R))}$.
However, Stefan Friedl~\cite{friedl} found a  gap in the proof, appearing  in the second to last sentence on page six of~\cite{G2}. 
We now want to state  a   corrected version of~\cite[Theorem 1]{G2}.

\begin{theorem}\label{sim} Assume the notations and suppositions of Theorem~\ref{thmcg}, and let $R$ and  $H$ be as above.
Let $p$ be a prime relatively prime to $|\Tors(H_1(R))|.$ Let $\{x'_i\}$ be a basis for $H$. Let  $\{y'_i\}$ be a complementary dual basis in $H_1(F)$ to $\{x'_i\}$, with respect to the intersection pairing.  View $F$ as built from a disk by adding $2g$ bands, with cores representing the $x'_i$ and $y'_i$.  Let the linking circles to those bands be denoted $x_i$ and $y_i$. 
Let $\mathbb{Y}$  be the subgroup of $H_1(S_q)$ generated by the lifts  of the $y_i$ to a single component of the inverse image of $S^3 \setminus F$ in $S_q.$ Then $\kappa(\Delta,q)_{(p)}= {\mathbb{Y}}_{(p)}$.
\end{theorem}

Two independent proofs of Theorem \ref{sim} are presented in Appendices~\ref{hlemma} and~\ref{sec2ndapproach}.  In \cite[Theorem 8.6]{friedl} and  \cite[page 511]{COT}, an equivalent  result is asserted for almost all primes $p$ (rather than for all primes not dividing  $|\Tors(H_1(R))|.$)

To each element   $z \in  H_1(S_q)_{(p)}$,  there is an associated character 
$$ \chi_z \co H_1(S_q)_{(p)} \to \zz_{p^k} \subset \qq / \zz ,$$  
(for some value of $k$) 
defined by $\chi_z(w)= \lk(w,z) \in \qq / \zz$. 

\begin{corollary} \label{simcor}Assuming  the notations and suppositions of Theorems~\ref{thmcg} and~\ref{sim}, then $\mathfrak{m}(q,\Delta)_{(p)}= \{\chi_z| z \in {\mathbb{Y}}_{(p)}\}.$
\end{corollary}

We can now summarize the proof of Theorem~\ref{thmwitt1}.  Details follow as in~\cite{G2}.
\begin{proof}[Proof of Theorem~\ref{thmwitt1}] 

By Theorem~\ref{sim}, one needs to show that the vanishing of the Casson-Gordon invariants for characters  $\chi_{z}$ with $z\in {\mathbb{Y}}_p$ implies the surgery curve $J$ satisfies the specified $(m(K),p)$--Witt conditions.  There are two steps.  First, one considers a new knot, $K'$, formed from $K$ by tying a knot  $-J$ in the band of the Seifert surface representing $J$.    This new knot is slice, since it has surgery curve $J \# -J$, which is slice.  The manifold $R$ for $K'$ is built by adding a two-handle to $F \times [0,1]$, and can be seen to be a solid handlebody, in fact, a solid torus.  Thus, Theorem~\ref{sim} implies that for all the relevant characters, the Casson-Gordon invariants vanish.  The proof is completed by proving that the effect of changing $K$ to $K'$ on the Casson-Gordon invariants is to add the  sum of invariants appearing in the $(m(K),p)$--Witt conditions.  
\end{proof}
   
\noindent(We take this opportunity to remark that  \cite[Theorem  (3.5)]{G1} remains valid. Although  \cite[Theorem 1.1]{G1} (which is \cite[Theorem 1]{G2} in the case $q=2$) is used in the proof
 of \cite[Theorem  (3.5)]{G1}, \cite[Theorem 1.1]{G1}  is only used in the case that $R$ is a handlebody. For similar reasons, the proof of  \cite[Theorem 7] {naik} is valid.)

\begin{proof}[Proof of Theorem~\ref{thmwitt2}] 
If $K$ is an algebraically slice knot of genus one,   $m = m(K)$, and $q$ is odd, then $H_1(S_q)$ is the direct sum of two cyclic groups of order $(m+1)^q -m^q$.
For each odd prime $p$ such that $ \gcd(p^2 ,  (m+1)^q-(m)^q)=p$, the $p$--primary part of $H_1(S_q)$
(denoted $H_1(S_q)_{(p)}$) is a two-dimensional vector space over $\zz_p$.
An analysis of   $H_1(S_q)$ (as in the proof of~\cite[top of page 16]{G2}) shows that the two metabolizers for the Seifert form spanned by the two surgery curves, say $J_1$ and $J_2$, lead to two distinct metabolizers
for the linking form restricted to $H_1(S_q)_{(p)}$. In fact, these metabolizers are eigenspaces for a generator of the group of covering transformations with the distinct eigenvalues $(m+1)/m$, $m/(m+1)$. Thus this linking form on $H_1(S_q)_{(p)}$ is hyperbolic. It follows that an  element in $H_1(S_q)_{(p)}$ in the complement of the union of these two metabolizers cannot have self-linking zero. So the linking form on $H_1(S_q)_{(p)}$ has only these two metabolizers. 
 
If $K$ is slice, then $\kappa(\Delta,q)_{(p)}$ must be one of these two metabolizers. 
Thus by Theorem~\ref{thmcg},  if $\chi \co H_1(S_q) \rightarrow (1/p) \zz/\zz$ vanishes on $\kappa(\Delta,q)_{(p)}$, then $\tau(K, \chi)=0$.   By~\cite[proof of Theorem 3]{G2}, for each of these  $p$, either  $J_1$ or  $J_2$ must satisfy the 
$(m,p)$--Witt conditions. But for 
  $K= W(J, (m(m+1))$, both $J_1$ and $J_2$ have the isotopy type of $J \# T_{(m,m+1)}$.
\end{proof}


\section{The averaging conditions restrict where the jumps can occur}\label{restrict}

We consider the family $\mathcal{J}$ of   step functions $f$ on $[0,1]$ which vanish at $0$ and $1$ and have a finite number of jumps, with value at the jumps the average of the one-sided limits.  
Define for $f \in \mathcal{J}$, $$\Sigma_p(f)=  \Sigma_{i=1}^{p-1}f(i/p).$$
Consider, also,  the family of symmetric jump functions
$$\mathcal{S}= \{ f \in \mathcal{J}| f(x)=f(1-x)\}.$$  These include the knot signature functions.

We  say $\sigma \in \mathcal{S}$ satisfies the  $m$--signature averaging condition if: for  each $p$ relatively prime to $m$ and $m+1$, $\Sigma_p(\sigma) =0$. 
The  $m$--signature averaging condition is a consequence of the  $(m,p)$--signature conditions for all $p$ relatively prime to $m$ and $m+1$.

The Alexander polynomial of the knot $5_2$ is $2 - 3 t + 2 t^2$  \cite{knotinfo} which has simple roots at
$\frac{1}{4} \left(3 \pm i \sqrt{7}\right)$. These roots lie on the unit circle and have argument $ \pm 2 \pi  a$ where 
$a=\frac 1 {2 \pi i} \log(\frac{1}{4} \left(3+i \sqrt{7}\right)) \approx
0.115$.

\begin{prop}\label{5_2} The  number  $a$ is irrational. The signature function of $ 5_2 \ \# -(5_2)_{2,1}$ has jumps in the interval 
$[0,\frac 1 2]$ at $\frac a 2$, $a$, and 
$\frac {1-a} 2$
 and this signature function satisfies the $(1,p)$--signature conditions for all odd $p$.
\end{prop}

\begin{proof} 
 If $a$ were rational,
$2 -  3 t + 2t^2$ would have to be a  factor of some cyclotomic polynomial; but these are monic.
The signature function of $5_2$ viewed as a function on $[0,1]$ has jumps at $a$ and $1-a$.
Using \cite{lith} or \cite{livmel}, the signature function of the knot  
$(5_2) \# - (5_2)_{2,1}$ jumps at exactly $\frac a 2$, $a$, $\frac {1-a} 2$,$\frac {1+a} 2$, $1-a$, $1-\frac a 2$.
 By Theorem
\ref{shifted1}, $(5_2)\# - (5_2)_{2,1}$ satisfies the  $(1,p)$--signature conditions for all odd $p$.
 \end{proof}

This example contradicts a claim that we once (see the  sentence beginning on the first line of~\cite[page 486]{GL1}) deferred to a future publication, but now retract.    
Note that  the locations of the irrational  jumps $\frac a 2$, $a$, $\frac {1-a} 2$ in the first half interval  together with  $1$ are linearly dependent  over $\qq$. Our next theorem says that this is necessary for the jumps of a signature function which satisfies the $m$--signature averaging condition.

For $0 < a < 1$, let $\chi_a$ denote the characteristic function which takes value  one on $[0,a)$, value  $1/2$ at $a$ and value zero on $(a,1]$. We have that
$$\Sigma_p(\chi_a)=  \begin{cases}
\lfloor pa \rfloor & \text{if $pa \notin \zz$},\\
\lfloor pa \rfloor -\frac 1 2 & \text{if $pa \in \zz$}.
\end{cases} $$ 
where  $\lfloor x \rfloor$ denotes the greatest integer in $x$.
 
For $0 < a < \frac 1 2$, consider the symmetric jump function on $[0,1]$, 
$S_a=  \chi_{1-a} -\chi_a$. We have that $S_a \in \mathcal{S}$ and   
$$\Sigma_p(S_a)= \lfloor p(1-a) \rfloor -\lfloor p a\rfloor .$$
We define  $F_p(a)$ by:
\begin{equation} \label{frac} 
\Sigma_p(S_a)-p \int_0^1 S_a(x)\ dx=     F_p(a)  = \begin{cases}
2 < p a> -1 & \text{if $pa \notin \zz$},\\
0 & \text{if $pa \in \zz$}.
\end{cases} 
\end{equation}
where   {\it $<x> = x-\lfloor x \rfloor$ denotes the fractional part of $x$}.

\begin{theorem}\label{thmirrational} 
Let $\sigma \in \mathcal{S}$  and let $\{ j_1, \ldots , j_s\}$ be  
the irrational points of discontinuity of $\sigma$ that lie in the interval $[0,\frac 1 2]$.
 Suppose $s \ge1$. 
If $\sigma$ satisfies the $m$--signature averaging condition, then $\{j_1, \ldots, j_s, 1\}$ are linearly dependent over $\qq$. \end{theorem}

\begin{proof} It is easily seen that the integral of $\sigma$ must be zero.
We assume that there is a jump at an  irrational point.  Thus $s\ge1$.  
 
We have that $\sigma$ can be written uniquely as
$\sum_{i=1}^r c_i S_{a_i}$  with the $c_i$ nonzero and the $a_i$ distinct. By reordering, 
we can assume that $a_i$ is rational if and only if $i > s$, for
some $s \le r$. Thus $\{ j_1, \ldots , j_s\}=\{ a_1, \ldots , a_s\}$. For each $i > s$, write $a_i = b_i/d_i$ in
lowest terms.   Let $D$ be the least common multiple of the elements of 
$\{ d_i \ | \  i>s \} \cup \{m, m+1\}$. 
Let $N =\{ p\  |\  p >0, p \equiv -1\pmod{D}\}$. For all $p \in N$, $\Sigma_p \sigma =0$, and $p a_i \notin \zz$. 
Hence, using~\ref{frac}, we have  that $\sum_{i=1}^r  c_i  < p a_i > = r/2$ for all $p\in N$. 
 
Since $p \in N$ is constant modulo
$D$, $\sum_{i=s+1}^r c_i <pa_i>$ is constant for $p \in N$. Hence the sum over the irrational terms,
$\sum_{i=1}^s c_i <pa_i> $ is constant for $p \in N$, as well. Thus
$$\mathcal{I}= \{(<p a_1>, <p a_2>, \cdots <p a_s>)\ | \ p \in N \}$$ is not dense in
$I^s$. 
KroneckerÔs Theorem~\cite[Theorem 442]{HW} states that if the fractional parts of the positive integral multiples of a vector 
$(a_1, a_2, \cdots a_s)$ are not dense in $I^s$, then $\{a_1, \ldots, a_s, 1\}$ are linearly dependent over $\qq$.  
It is not hard to see that the same holds for multiples by any arithmetic sequence, like $N$.
\end{proof}

The above theorem  still holds if one relaxes the hypothesis by removing the condition that the value of  $\sigma$ at the jump points be given by the average of  the one sided limits, as one could redefine  the values at these points without changing the values of $\Sigma_p (\sigma)$ for the specified $p$'s.

Note that, if $a$ is a rational whose denominator  divides $d$, then 
\begin{equation} \label{vary} F_p(a)=  F_{p+k d}(a)   =  - F_{-p+k d}(a). \end{equation}

\begin{definition}
 Given an odd number $d>1$, let $\mathbb{D}(d)$ be the determinant of the 
${\frac {d-1} 2} \times  {\frac {d-1} 2}$ matrix  indexed by $1\le i,j \le {\frac {d-1} 2}$ with entries
$$
 F_i(\frac  j d)= 
 \begin{cases}
2 < \frac {i j} d >-1 & \text{if $d \nmid i j $},\\
0 &  \text{if $d \mid i j$}. \end{cases}
$$ 
\end{definition}

\begin{conj}\label{con} For all odd numbers $d>1$, $\mathbb{D}(d)\ne 0.$
\end{conj}

This conjecture is true for $d$ prime according to the next proposition. We have verified the conjecture for $d< 1500$ using Mathematica. 
  
\begin{prop}\label{Dp} If ${s}$ is an odd prime, $\mathbb{D}({s})= \pm {2^{\frac {{s}-3} 2}}  h_{s}/{s}$, where $h_{s}$ is the first factor of the class number of the cyclotomic ring $\zz[\zeta_{s}]$. Thus $\mathbb{D}({s})\ne 0.$
\end{prop}
\begin{proof} 

The result follows from equations (1.7), (2.3), (2.4), and (2.5) of~\cite{carlO}.
\end{proof}

\begin{theorem} Let $d>1$ be a fixed odd integer for which  $\mathbb{D}(d) \ne 0$. Suppose $\sigma \in  \mathcal{S}$ has all jumps at rational points whose denominator divides $d$.   If $\Sigma_p ( \sigma ) =0$ for all odd $p$, then $\sigma =0$.
\end{theorem}

\begin{proof} We have  $\sigma= \sum_{j=1}^{\frac {d-1} 2} a_j S_{j/d}$ for some $a_j$.
Since $\Sigma_p ( \sigma ) =0$ for all odd $p$, we have that $\int_0^1 \sigma(x) dx=0$.
 We  pick an  odd integer $p(i)$ congruent to $i$ modulo $p$ for every $i$ in the range: $0 \le i\le  \frac {d-1}2$.  
 For each $i$, $\Sigma_{p(i)} ( \sigma ) - p(i) \int_0^1 \sigma(x) dx=0$. Using equations \ref{frac} and \ref{vary}, this gives us the linear equation 
$\sum_{j=1}^{\frac {d-1} 2} a_j F_i(\frac j d)=0.$ The resulting system of $\frac {d-1} 2$ equations in the $\frac {d-1} 2$ unknowns $a_j$ has only the trivial solution if $\mathbb{D}(d) \ne 0$.
\end{proof}

\begin{corollary} Suppose $d>1$ is an odd integer, and   $\mathbb{D}(d) \ne 0$. A non-zero knot signature function satisfying the $1$-signature averaging condition cannot   
have jumps only at points with denominator a divisor of $d$.
\end{corollary}

Since knot signature functions cannot jump at points with prime denominators~\cite{tristram}, Proposition~\ref{Dp} does not say anything about knots, except to the extent that it makes Conjecture~\ref{con} plausible.


\begin{appendix}

\section{Witt invariants of cable knots}\label{cableappendix}

The proof of Theorem~\ref{companionthm} follows fairly readily from work of Litherland, some basic knot theoretic results, and consideration of Witt groups.

We begin with an observation: if  $\mathbb{S}$  
is a satellite of $K$ with orbit $P$ and winding number $n$, then for an appropriate choice of Seifert surfaces for $K$,  $P$, and  $\mathbb{S}$, the Seifert matrix for   $\mathbb{S}$ is the direct sum of a Seifert matrix for $P $ and one for $C_{n,1}$.  The construction of the  Seifert surfaces  for a satellite knot which leads to the above result was first done by Seifert~\cite{seifert}.

Thus, to prove Theorem~\ref{companionthm} we need only prove the following:

\begin{theorem}
For $C_{(n,1)}$, the $(n,1)$--cable of $C$, $$ w_{C_{(n,1)}}(\frac{j}{p})= w_{C}(\frac{n j}{p}).$$
\end{theorem}  
\begin{proof}

The proof is largely contained in a diagram; note in the following description that the central square of the diagram is not apparently commutative, while 
 one has commutativity around the other interior faces of the diagram.

\begin{center}
$
\begin{diagram}
\dgARROWLENGTH=1em
\node{\mathcal{C}}\arrow{s,r}{\lambda_n} \arrow{e,t}{ \alpha}  \node{\mathcal{G}} \arrow{e,t}{\beta}\arrow{s,r}{\lambda'_n} \node{W(\qq[t,t^{-1}]_{(\phi_p)})}  \arrow{e,t}{\gamma}\arrow{s,r}{\eta'_n}  \node{W(\qq(t) )}\arrow{s,r}{\eta_n}   \\
\node{\mathcal{C}} \arrow{e,t}{ \alpha}  \node{\mathcal{G}} \arrow{e,t}{\beta}\arrow{se,r}{\rho'}\node{W(\qq[t,t^{-1}]_{(\phi_p)})}\arrow{e,t}{\gamma}\arrow{s,r}{\rho} \node{W(\qq(t) )}   \\
\node{  }   \node{    }   \node{W(\qq(\zeta_p) )}   \node{} 
\end{diagram}
$
\end{center}

\noindent Here is the notation and necessary background:

 \begin{itemize}
\item   $\mathcal{C}$ is the concordance group;  $\mathcal{G}$ is Levine's algebraic concordance group of Seifert matrices; $\alpha$ is the homomorphism induced by $K \to V_K$.\vskip.05in 

\item $W(\qq[t,t^{-1}]_{(\phi_p)}$  is the Witt group of the localization of $\qq[t,t^{-1}]$ at the $p$--cyclotomic polynomial $\phi_p$,  (that is, the domain formed by inverting all polynomials relatively prime to $\phi_p$); $\beta$ is the map induced by $$V\to (1-t)V + (1-t^{-1})V^t.$$\vskip.05in 
 \item 
 $W(\qq(t))$ is the Witt group of the field of fractions of $\qq[t,t^{-1}]$; $\gamma$ is induced by inclusion.  The inclusion map is injective (see~\cite[Corollary IV 3.3]{mh} in the symmetric case, and~\cite[Proposition 4.2.1 iii)]{ranick} for the hermitian case that arises here).\vskip.05in
 
 \item $\lambda_n$ is the {\it function} induced by forming the $(n,1)$--cable; $\lambda'_n$ is the homomorphism induced by $\lambda_n$.  This map can be given explicitly in terms of Seifert matrices.  That this induces a map on $\calg$ and that the map is a homomorphism is elementary.  (See~\cite{chalivrub, kawauchi} for further discussion.)\vskip.05in
 
 \item The map $\rho$ is induced by the map $t \to \zeta_p$.\vskip.05in
 
  \item The map $\eta_n$ ( respectively $\eta'_n$) is induced by the embedding of  $\qq(t)$ ( respectively $ \qq[t,t^{-1}]_{(\phi_p)}$ )  into  itself which sends $t$ to $t^n$.

  \end{itemize}
  
  The proof of Theorem~\ref{companionthm} is seen to be equivalent to showing that  $$ \rho' \circ \alpha \circ \lambda_n = \rho \circ \eta_n' \circ \beta \circ \alpha.$$ By writing $\rho' = \rho \circ \beta$, we see this will follow from 
 $$ \beta \circ \alpha \circ \lambda_n = \eta_n' \circ \beta \circ \alpha.$$
 
 According to Litherland~\cite{lith}, we have 
  $$\gamma \circ \beta \circ \alpha \circ \lambda_n = \eta_n\circ \gamma \circ \beta \circ \alpha.$$ Using commutativity of the rightmost square, we have $\eta_n \circ \gamma = \gamma \circ \eta_n'$, so Litherland's equality can be rewritten  as 
 $$\gamma \circ \beta \circ \alpha \circ \lambda_n =\gamma \circ \eta_n' \circ \beta \circ \alpha.$$
Finally, because $\gamma$ is injective, this implies 
$ \beta \circ \alpha \circ \lambda_n =  \eta_n' \circ \beta \circ \alpha,$ as desired.
 \end{proof}

\section{One approach to Theorem~\ref{sim}}\label{hlemma}

Let $\qq' =\{ r/s  \in \qq  | \gcd(s,r)=\gcd(s, |\Tors(H_1(R))| )=1\}.$

\begin{lemma} \label{tor}
If $T$ is a finitely generated torsion group, and the prime divisors of $|T|$ are all divisors of   $|\Tors(H_1(R))|$, then $T \otimes (\qq'/\zz)=0$, and 
$\Tor (T,\qq'/\zz)=0$.
\end{lemma}
\begin{proof} It suffices to prove this for $T$ a finite cyclic group of order $k$ relativley prime to all the denominators of elements of  $\qq'.$
From the short exact sequence; $$ 0 \rightarrow \zz  \stackrel{  k \cdot}{\rightarrow} \zz  \rightarrow T   \rightarrow   0,$$ 
we obtain: 
$$0  \rightarrow \Tor (T,\qq'/\zz) \rightarrow   \qq'/\zz     \stackrel{  k \cdot}{\rightarrow} \qq'/\zz   \rightarrow T \otimes \qq'/\zz   \rightarrow   0 .$$  
Suppose $s$ is a denominator of an element in $\qq'$, then $\gcd(k,s)=1$, and there exists $a, b \in \zz$, such that $ka+sb=1$.  It follows that $k \cdot  a/s  \equiv  1/s \pmod{1}$. Thus  $ \qq'/\zz \stackrel{  k \cdot}{\rightarrow} \qq'/\zz$ is  surjective. It is easy to see that  $ \qq'/\zz \stackrel{  k \cdot}{\rightarrow} \qq'/\zz$ is  injective. 
\end{proof}

\begin{lemma} \label{short}
A short exact sequence of the form  $$ 0 \rightarrow  T_1
 \stackrel{ \psi} 
\rightarrow T_2 \oplus F_2
 \stackrel{ \phi}   \rightarrow T_3 \oplus F_3  \rightarrow 0 ,$$
  where the $F_i$ are free abelian groups, and the $T_i$ are torsion groups, 
  induces a short exact sequence:
  $$ 0 \rightarrow  T_1 
 \stackrel{\pi_{T_2} \circ \psi } 
\rightarrow T_2 
 \stackrel{ \phi_{| T_2}} \rightarrow T_3  \rightarrow 0 .$$ 
\end{lemma}

\begin{proof} Exactness on the left, and at the middle of this sequence is immediate. We only need to show 
that $\phi_{| T_2}$ is surjective. Let $x \in T_3$, there exist $(y, z) \in T_2 \oplus F_2$ with $\phi((y, z))= x$.
We wish to show that $z=0$.
There exist  nonzero integers $n$ and $m$  such that $nx=0$, and $my=0$. Then  $\phi((0, mnz))=\phi((mny, mnz))= mnx=0$.
By exactness of the original sequence, $(0, mnz) \in \psi(T_1)$. Since $z \in F_2$, we have that $z=0.$ 

\end{proof}

\begin{lemma} \label{hlemma2}
Let  $\mathcal{T}$ denote $\Tors(H_1(R))$,  and let $H$ denote  the kernel of   $H_1(F) \rightarrow H_1(R)/\mathcal{T}$. We have that 
$H \otimes \qq'/\zz$ is the kernel of the natural map  $H_1(F) \otimes (\qq'/\zz) \rightarrow H_1(R) \otimes (\qq'/\zz).$
\end{lemma}

\begin{proof}  
Let ${\mathcal{I}}$ be the image of $H_1(F) \rightarrow H_1(R)$, and  $\hat {\mathcal{I}}$ be the image of $H_1(F) \rightarrow H_1(R)/\mathcal{T}$. 
We have a short exact sequence:
$$ 0 \rightarrow H  \rightarrow H_1(F)  \rightarrow \hat {\mathcal{I}}   \rightarrow   0.$$
As $\hat {\mathcal{I}}$ is free abelian, $\Tor(\hat {\mathcal{I}}, \qq'/\zz)=0$,  and
we then have a short exact sequence:
$$ 0 \rightarrow H \otimes (\qq'/ \zz) \rightarrow H_1(F) \otimes (\qq'/ \zz)  \rightarrow \hat {\mathcal{I}} \otimes (\qq'/ \zz)  \rightarrow 0.$$

Let $\mathcal{R}$ denote $H_1(R)$, and note that $\mathcal{I} / ( \mathcal{I} \cap  \mathcal{T} )=  \hat {\mathcal{I}}.$
 Consider the lattice of subgroups consisting of $\mathcal{R}$, $\mathcal{I}$,  $\mathcal{T}$,  $\mathcal{I}\cap \mathcal{T}$.
Their inclusions fit into the following commutative diagram with exact rows and columns:
\begin{center}
$
\begin{diagram}
\dgARROWLENGTH=1em
\node{ }  \node{0} \arrow{s,r}{ } \node{0} \arrow{s} \node{0} \arrow{s} \node{}\\
\node{ 0} \arrow{e,t}{ }  \node{{ \mathcal{I}} \cap   \mathcal{T}} \arrow{e}\arrow{s}\node{ \mathcal{T}} \arrow{e}\arrow{s} \node{\mathcal{T}/( {\mathcal{I}} \cap   \mathcal{T} ) } \arrow{e}\arrow{s} \node{0} \\
\node{ 0} \arrow{e } \node{{ \mathcal{I}}    }\arrow{e}\arrow{s} \node{  \mathcal{R}  } \arrow{e}\arrow{s} \node{  \mathcal{R}/\mathcal{I} } \arrow{e}\arrow{s} \node{0}\\
\node{ 0} \arrow{e } \node{ { \hat {\mathcal{I}}}  } \arrow{e}\arrow{s} \node{ \mathcal{R}/\mathcal{T} } \arrow{e}\arrow{s} 
\node{   ( \mathcal{R}/\mathcal{T} ) / \hat {\mathcal{I}}}
\arrow{e}\arrow{s} \node{0}\\
\node{}   \node{0}   \node{0}   \node{0}   \node{}
\end{diagram}
$
\end{center}
To see exactness, view the first two columns as the inclusion of one chain complex into another. The third column is the quotient chain complex. Thus we have a short 
exact sequence of chain complexes. The first two chain complexes are clearly exact. It follows that the third column is exact, using the associated long exact sequence of homology groups.

Using the long exact sequence of the pair $(R,F)$, we may identify $\mathcal{R}/\mathcal{I}$ with $H_1(R,F)$.  Using Lefschetz duality and the universal coeficient theorem, we have   $H_1(R,F) \approx H^2(R,\Delta) \approx H^2(R)  \approx \mathcal{T} \oplus \zz^{\beta_2(R)}$. 
With  these identifications, the last column of the diagram becomes a short exact sequence:
$$ 0 \rightarrow  \mathcal{T}/( {\mathcal{I}} \cap   \mathcal{T} )  
\rightarrow  \mathcal{T} \oplus \mathcal{F} 
   \rightarrow  \Tors (  (  \mathcal{R}/\mathcal{T}) / \hat {\mathcal{I}} ) \oplus \mathcal{F'} \
  \rightarrow 0 ,$$
  where $\mathcal{F}$ and $\mathcal{F}'$ are free abelian groups. 
  By Lemma \ref{short},
   there is a short exact sequence:
  $$ 0 \rightarrow  \mathcal{T}/( {\mathcal{I}} \cap   \mathcal{T} )  
 \rightarrow  \mathcal{T} 
\rightarrow  \Tors (  (  \mathcal{R}/\mathcal{T}) / \hat {\mathcal{I}} ) \
  \rightarrow 0 ,$$
We conclude that 
$|\Tors ( \mathcal{R}/\mathcal{T} ) / \hat {\mathcal{I}})| = |  {\mathcal{I}} \cap   \mathcal{T}|$.  We have that
$$\Tor( ( \mathcal{R}/\mathcal{T} ) / \hat {\mathcal{I}},\qq'/ \zz)= \Tor( \Tors(( \mathcal{R}/\mathcal{T} ) / \hat {\mathcal{I}}),\qq'/ \zz)= 0,$$
by Lemma \ref{tor}. 
So the sequence obtained from the last row of the diagram upon tensoring with 
$\qq'/ \zz$ is exact. In particular,  the map
$ \hat {\mathcal{I}} \otimes (\qq'/ \zz)  \rightarrow (H_1(R)/\mathcal{T})  \otimes (\qq'/ \zz)$ is injective. 
It follows that $H \otimes (\qq'/ \zz)$, the kernel of  $H_1(F) \otimes (\qq'/ \zz)  \rightarrow \hat {\mathcal{I}} \otimes (\qq'/ \zz)$,
 is the same as the kernel of $H_1(F) \otimes (\qq'/\zz) \rightarrow (H_1(R)/\mathcal{T}) \otimes (\qq'/\zz)$.
 
  Considering the middle column, we obtain the  following  exact  sequence:
$$  \mathcal{T} \otimes (\qq'/ \zz) \rightarrow H_1(R)  \otimes (\qq'/ \zz)  \rightarrow H_1(R)/\mathcal{T}  \otimes (\qq'/ \zz)  \rightarrow 0.$$ 
Since by Lemma \ref{tor},   $ \mathcal{T} \otimes (\qq'/ \zz)=0$, 
  we see that $H_1(R)  \otimes (\qq'/ \zz)  \rightarrow (H_1(R)/\mathcal{T})  \otimes (\qq'/ \zz)$
 is injective.  
Thus  the kernel of  $H_1(F) \otimes (\qq'/\zz) \rightarrow (H_1(R)/\mathcal{T}) \otimes (\qq'/\zz)$ is also the kernel of $H_1(F) \otimes (\qq'/\zz) \rightarrow H_1(R) \otimes (\qq'/\zz)$.
\end{proof}

The second to last sentence of~\cite[page 6]{G2} asserts without justification, in the situation of Lemma \ref{hlemma2}, that $H \otimes \qq/\zz$ is the kernel of the natural map  $H_1(F) \otimes (\qq/\zz) \rightarrow H_1(R) \otimes (\qq/\zz)$. 
The original proof of~\cite[Theorem 1]{G2} may then be modified  using  Lemma \ref{hlemma2} and replacing $\qq/\zz$ by  $\qq'/\zz$ 
judiciously. 
This  proof  then yields the conclusion:  ${A^q}_{p} \cap (H \otimes \qq/\zz)$ (in the notation of \cite{G2}) is equal to  $ \mathfrak{m}(q,\Delta)_{(p)}$, for primes $p$ relatively prime to $|{\Tors(H_1(R)}|$. This in turn can be rephrased as Theorem~\ref{sim}.


\section{Another approach to Theorem~\ref{sim}}\label{sec2ndapproach}

\subsection{Notation}

\begin{itemize}
\item $K$  is a slice knot with genus $g$ Seifert surface $F$; $K$ bounds a slice disk $\Delta$; $R \subset B^4$ is a 3--manifold bounded by $F \cup \Delta$.\vskip.05in

\item $S_q$ is the $q$--fold branched cover of $S^3$ branched over $K$ and $W_q$ is the $q$--fold branched cover of $B^4$ branched over $\Delta$.\vskip.05in

\item $H$ is the kernel of $H_1(F) \to H_1(R)/ \text{Torsion}(H_1(R))$;  $\kappa(q,\Delta)$ is the kernel of $H_1(S_q) \to H_1(W_q)$. \vskip.05in

\end{itemize}

We further choose generators for various homology groups:

\begin{itemize}

\item  $\{x'_i\} \cup \{y'_i\}$ is a 
 symplectic basis of $H_1(F)$   where the $x_i$ generate $H$.  \vskip.05in
 
 \item   $F$ is    built from a disk with   1--handles added corresponding to this basis.  The dual linking circles to the bands represent homology classes in $H_1(S^3 \setminus F)$  denoted $\{x_i\} \cup \{y_i\}$.

\end{itemize}

  Recall (see~\cite{rolfsen}) that $S_q$ is built from $q$ copies of $S^3 \setminus F$. These copies can be enumerated cyclically, corresponding to translates under the  deck transformation.  There is a corresponding enumeration of the lifts of $F$  to $S_q$.
  
  \begin{itemize}
  \item The lifts of the $x_i$  are denoted $\tilde{x}_{i,\alpha}$, and similarly for the $\tilde{y}_i$, $\tilde{x}'_i$ and $\tilde{y}'_i$.   The $\alpha$ are indices denoting the appropriate lift of $ S^3 \setminus F$ and $F$.  Here, $\alpha \in \zz_q$.
  \vskip.05in

 \item $\mathcal Y$ denotes the subgroup of $H_1(S_q)$ generated by the $\tilde{y}_{i,\alpha}$.   Similarly for $\mathcal X$, $\mathcal Y'$, and $\mathcal X'$. \vskip.05in

    \item $\mathbb Y$ denotes the subgroup of generated by a single set of lifts:  $\{\tilde{y}_{i,0}\}$.   \vskip.05in
  
  \end{itemize}

\subsection{Statement and proof summary}

Theorem~\ref{sim}   can now be stated succinctly:  if  $p$ relatively prime to the order of Torsion($H_1(R))$, then $\kappa(\Delta, q)_{(p)} = {\mathbb Y}_{(p)}$.  The proof has several steps:

\begin{itemize}

\item Lemma~\ref{relateyxlemma}:     $H_1(S_q)_{(p)} = {\mathcal Y}_{(p)} \oplus {\mathcal X}_{(p)}$ and $|{\mathcal Y}_{(p)} |= |  {\mathcal X}_{(p)}|$.\vskip.05in

\item Lemma~\ref{dualxequalylemma}:   $   {\mathcal X}'_{(p)} = {\mathcal Y}_{(p)} $.\vskip.05in

\item Lemma~\ref{xkappalemma}: $   {\mathcal X}'_{(p)} \subset   \kappa(\Delta, q)_{(p)}$ and $|  \kappa(\Delta, q)_{(p)}|^2 = |H_1(S_q)_{(p)}|$.\vskip.05in

\item Lemma~\ref{localizelemma} ${\mathbb Y}_{(p)} = \mathcal Y_{(p)}$.   \vskip.05in

\end{itemize}

\noindent{\bf Proof of Theorem~\ref{sim}.}  We want to show that   $\kappa(\Delta, q)_{(p)} = {\mathbb Y}_{(p)}$.  By Lemma~\ref{localizelemma}, this is equivalent to showing that  $\kappa(\Delta, q)_{(p)} ={\mathcal Y}_{(p)}$.  By Lemmas~\ref{relateyxlemma} and~\ref{xkappalemma}, the orders of these two groups are the same.  By Lemmas~\ref{dualxequalylemma} and~\ref{xkappalemma}, ${\mathcal Y}_{(p)} \subset \kappa(\Delta,q)$, and the proof is complete.

\subsection{Proofs of lemmas}

\begin{lemma}\label{relateyxlemma}   $H_1(S_q)_{(p)} = {\mathcal Y}_{(p)} \oplus {\mathcal X}_{(p)}$ and $|{\mathcal Y}_{(p)} |= |  {\mathcal X}_{(p)}|$.

\end{lemma}
\begin{proof}

We use the convention that the Seifert form $V$ is the pairing $V(a,b) = \text{link}(i_+(a),b )$, where $i_+$ is the positive push-off.  For transformations we have matrices acting on the left, and in presentation matrices, the rows give the relations.   

The Seifert matrix of $V$ for the surface $F$  with respect to the basis $\{x'_i\} \cup \{y'_i\}$ for $H_1(F)$ is of the form $$
\left(
\begin{array}{cc}
0  &  M    \\
M^t  +I &   B  
\end{array}
\right)
$$
for some $g$ dimensional square matrices $M$ and $B$, with $B$ symmetric.
 
The first homology of $S_q$ is generated by (all) the lifts of the 
$x_i$ and $y_i$, which we have denoted $\tilde { x}_{i,\alpha}$ and $\tilde{y}_{i,\alpha}$.  As described, for instance in~\cite[page 213]{rolfsen},  a presentation matrix of the first homology of $S_q$ with respect to this basis  is determined by $V$.  In this case the result is a matrix of the form   
$$\left(
\begin{array}{cc}
0  &  \mathcal{ M} \\
\mathcal{M'} & \mathcal{B}      
\end{array}
\right)
$$ 
where $\mathcal{M}$ and $\mathcal{B}$ are $qg$ dimensional matrices that are built out of the blocks of $V$
like so (we illustrate in the case $q=3$):
$$
 \mathcal{ M}=
\left(
\begin{array}{cccc}

 M + I  &-M&0 \\
0  &  M+I  &-M\\
-M &0 &M+I   \\    
\end{array}
\right)     ,  \  \mathcal{ M'}=
\left(
\begin{array}{ccc}

  M^t & -M^t  -I  & 0 \\
 0&  M^t  & -M^t  -I   \\
 -M^t  -I & 0 &  M^t   
\end{array}
\right) \ ,
 $$

$$
 \mathcal{B}=
\left(
\begin{array}{cccc}

 B & -B &0\\
0&  B & -B    \\
-B  &0 &  B \\
\end{array}
\right).
$$
The first columns correspond to the $\tilde{x}_{i,\alpha}$ and the later columns to the $\tilde{y}_{i,\alpha}$. 

 Notice first that $|{\mathcal M|} = | {\mathcal M'|}$  and   $|{\mathcal M|}^2 = |H_1(S_q)|$. 
 
Forming the quotient, $H_1(S_q)/  {\mathcal Y}$ yields a group $\overline {\mathcal X}$    generated by the image of $\mathcal X$.  This quotient is presented by $\mathcal M '$ and thus has order $\sqrt{|H_1(S_q)|}$, so $\mathcal X$ has order at least this large.  Thus, $ | {\mathcal X}_{(p)}|^2 \ge 
|H_1(S_q)_{(p)}|$.  On the other hand, since 
$ {\mathcal X}_{(p)} $ is a self-annihilating  subgroup for a nonsingular form,  $ | {\mathcal X}_{(p)}|^2 \le 
|H_1(S_q)_{(p)}|$.  

We now have $    | {\mathcal X}_{(p)}|^2 =    
|H_1(S_q)_{(p)}|$ and thus $    | {\mathcal X}_{(p)}|  =      |\overline {\mathcal X}_{(p)}| $.  From this we can conclude that ${\mathcal X}_p \cap {\mathcal Y}_{(p)}  =0$, so $H_1(S_q)_{(p)}  =   {\mathcal X}_{(p)}  \oplus {\mathcal Y}_{(p)}$.

\end{proof}
  
\begin{lemma}\label{dualxequalylemma} $   {\mathcal X}'_{(p)} = {\mathcal Y}_{(p)} $.
\end{lemma}
\begin{proof}   The positive and negative push-off maps $i_{\pm}: H_1(F) \to S^3 \setminus F$ send the span of the $x_i'$ to the span of the $y_i$.  Denote the restriction of these maps by $j_{\pm}: \left<\{ x_i' \}\right > \to \left<\{ y_i \} \right>$. With respect to these bases, the maps $j_{\pm}$ are given by the matrices  $M^t$ and $M^t+I$.  Now view these matrices as defining maps from $\zz^g$ to itself with $M^t$ corresponding to an automorphism $T$. Then any element $y \in   \zz^g$ can be written  $y = \text{Id}(y) = (T+Id)(y) - T(y)$.  Thus, Image($j_+$) + Image($j_-) = $ Span$(\{y_i\})$.  Lifting to the $q$--fold branched covers, we see that the $\tilde{y}_{i,\alpha}$ are all in the image of the $\tilde{x}'_{i,\alpha}$. (In more detail, each $  \tilde{y}_{i,\alpha } $ is in the span of the images of the $\{\tilde{x}'_{i,\alpha } \}$ and $ \{\tilde{x}'_{i,\alpha +1 } \} $.)   Also, the images of the  $\tilde{x}'_{i,\alpha}$ are all in   Span$(\{\tilde{y}_{i,\alpha }\})$.  The same thus holds on the level of the $p$--torsion, completing the proof of the lemma.

\end{proof}

\begin{lemma}\label{xkappalemma}  ${\mathcal X}'_{(p)} \subset   \kappa(\Delta, q)_{(p)}$ and $|  \kappa(\Delta, q)_{(p)}|^2 = |H_1(S_q)_{(p)}|$.

\end{lemma}
\begin{proof}
Let $\gamma = |\text{Torsion}(H_1(R))|$.  Then for all $z \in H$,  $\gamma z = 0 \in H_1(R)$.  Lifting, we see that for all $z' \in {\mathcal X}'_{(p)}$, $\gamma z' = 0 \in H_1(W_q)$, so $\gamma  {\mathcal X}'_{(p)} \subset \kappa(\Delta, q)_{(p)}$.  But multiplication by $\gamma$ is an isomorphism on $   {\mathcal X}'_{(p)}$ since $p$ is relatively prime to $\gamma$.   

We have from Theorem~\ref{thmcg} that  $|  \kappa(\Delta, q) |^2 = |H_1(S_q) |$, so the same holds for the $p$--torsion.
\end{proof}

\begin{lemma}\label{localizelemma}${\mathbb Y}_{(p)} = \mathcal Y_{(p)}$

\end{lemma}
\begin{proof}  Let $\Lambda = \zz[\zz_q]$, the group ring of the cyclic group. We write $\zz_q$ multiplicatively, generated by $t$.  The standard derivation of a presentation of the homology $H_1(S_q)$, such as in~\cite{rolfsen}, is a Mayer-Vietoris argument.  The homology groups involved are all modules over $\Lambda$, where $t$ acts by the deck transformation.  From this viewpoint, the Mayer-Vietoris sequence now yields that as a $\Lambda$--module the homology is given as a quotient $H_1(S_q) \cong \Lambda^{2g} / (V-tV^t)\Lambda^{2g}$.  

Since $V- V^t$ is invertible, we can multiply the quotienting submodule by $(V - V^t)^{-1}$ without changing the quotient space.  Some elementary algebra then shows that $$H_1(S_q) \cong  \Lambda^{2g} /( \Gamma + t(I - \Gamma)) \Lambda^{2g},$$
where $\Gamma = (V -V^t)^{-1} V$.

It is clear from this that for any $z \in \Lambda^{2g}$ we have $ \Gamma z =  t ( \Gamma -I) z \in H_1(S_q)$.  Thus, $ \Gamma^q z  =  t^q(  \Gamma -I)^q z  \in H_1(S_q)$.  However, $t^q = 1$, so $\Gamma^q - (\Gamma - I)^q$ annihilates $H_1(S_q)$.

Expanding, we have that for some polynomial $f$ with 0 constant term and of degree $q-1$, $f(\Gamma) = I$ acting on $H_1(S_q)$.  The leading coefficient of $f$ is $q$.  If $p$ does not divide the order $|H_1(S_q)|$, the lemma is immediately true, so assume $p$ divides  the order $|H_1(S_q)|$.  We know that $p$ is relatively prime to $q$.  Thus, we can switch to $\zz_{(p)}$--coefficients, in which case the leading coefficient of $f$ is a unit, and we see that with $\zz_{(p)}$--coefficients, $\Gamma$ is invertible.

We now focus on the Siefert matrix $V$ of the algebraically slice knot.  In the coordinates we have been using, we see that 
$$\Gamma = \left(
\begin{array}{cc}
M^t  +I&  B \\
0   &   -M
\end{array}
\right).
$$
From this we conclude that with $\zz_{(p)}$--coefficients, $M$ and $M+I$ are both invertible.  Recall that for each $k$,  $M$ and $M+I$   determine the maps from  Span($\tilde{x}'_{i, k}$) and  Span($\tilde{x}'_{i, k+1}$)  to Span($\tilde{y}_{i,k}$).  Thus, any element in  Span($\tilde{y}_{i, k }$) is also in Span($\tilde{y}_{i, k+1}$).  This completes the proof of the lemma.

\end{proof}

\end{appendix}

\vskip.2in
\noindent{\it Acknowledgments\ } We thank John Ewing for his help with the number theoretic aspects of this project.  In particular, John pointed out to us the relationship between the determinant   $\mathbb{D}({s})$ and class numbers.  We also appreciate Stefan Friedl's careful reading and thoughtful commentary on our initial draft of this article. 


\newcommand{\etalchar}[1]{$^{#1}$}
 
\end{document}